\documentclass[12pt,letterpaper]{amsart}

\usepackage{epsfig,amsmath,amsthm,amssymb,graphicx}
\usepackage[top=1in, bottom=1in, left=1in, right=1in]{geometry}
\usepackage{color}

\usepackage[alphabetic]{amsrefs}

\usepackage{pb-diagram}
	
\newtheorem{proposition}{Proposition}[section]
\newtheorem{theorem}[proposition]{Theorem}
\newtheorem{corollary}[proposition]{Corollary}

\newtheorem{remark}[proposition]{Remark}

\newtheorem*{theorem*}{Theorem}
\newtheorem*{proposition*}{Proposition}
\newtheorem*{lemma*}{Lemma}
\newtheorem*{corollary*}{Corollary}

\newtheorem*{rep@theorem}{\rep@title}
\newcommand{\newreptheorem}[2]{
\newenvironment{rep#1}[1]{
 \def\rep@title{#2 \ref{##1}}
 \begin{rep@theorem}}
 {\end{rep@theorem}}}
\makeatother

\theoremstyle{definition}
\newtheorem{definition}[proposition]{Definition}

\newreptheorem{theorem}{Theorem}
\newreptheorem{lemma}{Lemma}
\newreptheorem{proposition}{Proposition}
\newreptheorem{corollary}{Corollary}

\newcommand{\bdry}{\partial}

\newcommand{\N}{\mathbb{N}}
\newcommand{\Z}{\mathbb{Z}}
\newcommand{\F}{\mathcal{F}}
\newcommand{\C}{\mathcal{C}}

\newcommand{\im}{\operatorname{Im}}
\newcommand{\lnk}{\operatorname{lk}}
\newcommand{\Id}{\operatorname{Id}}
\newcommand{\Arf}{\operatorname{Arf}}
\newcommand{\TL}{\operatorname{T.L.}}
\newcommand{\SL}{\operatorname{S.L.}}

\newcommand{\Sum}{\displaystyle \sum}

\begin{document}
\title[algebraic concordance and 1-solvability]{Every genus one algebraically slice knot is 1-solvable.}

\author{Christopher W.\ Davis}
\address{Department of Mathematics, University of Wisconsin--Eau Claire}
\email{daviscw@uwec.edu}
\urladdr{people.uwec.edu/daviscw}

\author{Taylor Martin}
\address{Department of Mathematics, Sam Houston State University}
\email{taylor.martin@shsu.edu}

\author{Carolyn Otto}
\address{Department of Mathematics, University of Wisconsin--Eau Claire}
\email{ottoa@uwec.edu}

\author{JungHwan Park$^\dag$}
\address{School of Mathematics, Georgia Institute of Technology}
\email{junghwan.park@math.gatech.edu}

\date{\today}
\thanks{$^{\dag}$ Partially supported by the National Science Foundation grant DMS-1309081.}


\begin{abstract} Cochran, Orr, and Teichner developed a filtration of the knot concordance group indexed by half integers called the solvable filtration.  Its terms are denoted by $\F_n$.  It has been shown that $\F_n/\F_{n.5}$ is a very large group for $n\ge 0$.  For a generalization to the setting of links the third author showed that $\F_{n.5}/\F_{n+1}$ is non-trivial.  In this paper we provide evidence that for knots $\F_{0.5}=\F_1$.  In  particular we prove that every genus 1 algebraically slice knot is 1-solvable.
\end{abstract}

\maketitle

\section{Introduction}

The knot concordance group $\C$, has been intently studied since its inception by Fox and Milnor in 1966 \cite{FM}.
 In \cite{COT03} Cochran-Orr-Teichner produced a filtration of $\C$ by subgroups
$$
\dots \le \F_{2}\le \F_{1.5}\le \F_{1}\le \F_{0.5}\le \F_0\le\C.
$$
Here, a knot in $\F_n$ is called \textbf{$n$-solvable}.  We recall the precise definition in subsection~\ref{sect:filtration}.

The first two terms of the $n$-solvable filtration are understood: $\F_0$ consists of knots whose Arf-invariants vanish and $\F_{0.5}$ consists of algebraically slice knots.  For all $n\in \N$ it is shown in \cite{CHL10} that ${\F_n}/{\F_{n.5}}$ contains an infinite rank free abelian subgroup.  For an analogous filtration $\F_n^m$ of the $m$-component (string) link concordance groups $\C^m$, the third author showed in \cite{Otto14} that the other half of the filtration ${\F_{n.5}}/{\F_{n+1}}$ is non-zero.  The analogous result is not known for knots;  indeed, our first theorem can be interpreted as suggesting that $\F_{0.5}=\F_1$.  

\begin{reptheorem}{main thm genus 1}
If a knot $K$ has genus $1$ and is algebraically slice, then $K$ is $1$-solvable.
\end{reptheorem}

In order to put this result into perspective as something surprising, we take a moment and recall the most common construction of knots which are highly solvable.  

If $K$ is an algebraically slice knot bounding a genus $g$ Seifert surface $F$, then on $F$ there exists a collection of $g$ disjoint nonseparating simple closed curves, $L_1,\dots,L_g$ for which the linking number $\lnk(L_i,L_j^+)$ vanishes for all $i,j\in \{1,\dots,g\}$.  Here, $L_j^+$ is the result of pushing $L_j$ off of $F$ in the positive normal direction.  Indeed, the existence of these curves is the definition of $K$ being algebraically slice.  Such a collection of simple closed curves is called a \textbf{derivative} for $K$.  
In \cite[Theorem 8.9]{COT03}, it is shown that an algebraically slice knot will always be more highly solvable than its derivatives, we recall the precise result here.

\begin{proposition*}[\cite{COT03}, Theorem 8.9]
If $L$ is a derivative for $K$ and $L$ is $n$-solvable then $K$ is $(n+1)$-solvable.  If $L$ is $n.5$-solvable, then $K$ is $(n+1.5)$-solvable.
\end{proposition*}

Until only recently \cite{CD15},  it was not known if the converse were true.  That is, if a knot is $(n+1)$-solvable, does it follow that it has an $n$-solvable derivative?  The following corollary of our Theorem~\ref{main thm genus 1} reveals that counterexamples are plentiful, at least in the case that $n=0$.  

\begin{corollary}
Let $K$ be the knot of Figure~\ref{fig:946example}, and let $J$ be a knot with $\Arf(J) = 1$.  Then $K$ is $1$-solvable and on the Seifert surface depicted in Figure~\ref{fig:946example}, $K$ does not admit a $0$-solvable derivative.
\end{corollary}
\begin{proof}
The curves $\alpha$ and $\beta$ of Figure~\ref{fig:946example} are each derivatives of $K$, showing that $K$ is algebraically slice.  Since $K$ also has a genus $1$ Seifert surface Theorem~\ref{main thm genus 1} implies that $K$ is $1$-solvable.  

The only two curves with zero self linking on $F$ are  depicted in Figure~\ref{fig:946example}.   Each has the knot type of $J$.  Thus, if $\Arf(J) = 1$, then $J$ is not $0$-solvable, and so $K$ admits no $0$-solvable derivative. \end{proof}

\begin{figure}
\setlength{\unitlength}{1pt}
\begin{picture}(95,100)
\put(0,0){\includegraphics[height=.17\textheight]{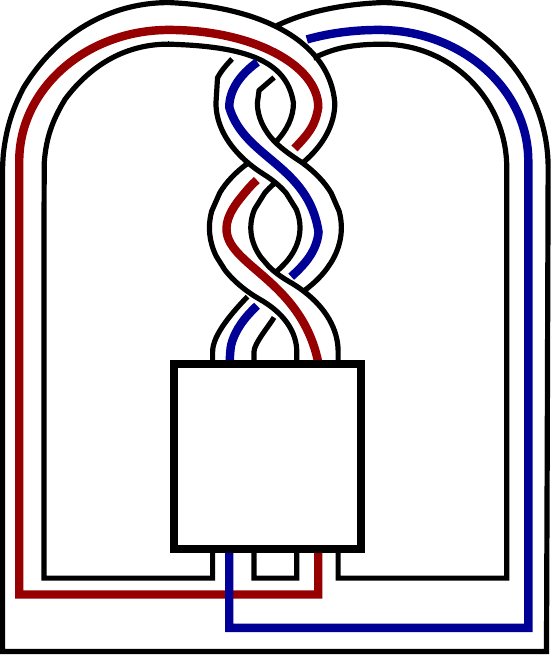}}
\put(40,28){$J$}
\put(76,19){\textcolor{blue}{$\alpha$}}
\put(12,19){\textcolor{red}{$\beta$}}
\end{picture}
\caption{
A surprising $1$-solvable knot together with its derivatives.}\label{fig:946example}
\end{figure}
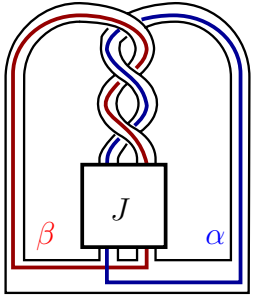

In order to describe our higher genus results, we will need to make reference to the Seifert matrix.  Let $L = L_1,\dots, L_g$ be a derivative for $K$ on the Seifert surface $F$.  This collection can then be extended to a basis for $H_1(F)$, $\{L_1,\dots,L_g, L_{g+1},\dots,L_{2g}\}$.  With respect to this basis, the Seifert matrix  $M=(m_{i,j})$ has $(i,j)$-entry $m_{i,j} = \lnk(L_i,L_j^+)$.  Since $L= L_1,\dots,L_g$ is a derivative, $M$ has the form 
$$
M=\begin{bmatrix}0&A\\B&C\end{bmatrix}
$$
where $0$ is the $g\times g$ zero matrix and $A$, $B$ and $C$ are all $g\times g$ matrices.  

\begin{reptheorem}{main thm genus 2}
Let $K$ be an algebraically slice knot with genus 2 Seifert surface $F$ and derivative $L=L_1,L_2$.   Extend $L_1,L_2$ to a basis for $H_1(F)$.  Let $M=\begin{bmatrix}0&A\\B&C\end{bmatrix}$ be the resulting Seifert matrix.

If either $\overline{\mu}_{1122}(L)$ is even or $\det(A)-\det(B)$ is odd, then $K$ is 1-solvable.  
\end{reptheorem}

The Sato-Levine invariant $\overline\mu_{1122}$ is an integer valued invariant of links with zero linking number.  We recall its definition in Section~\ref{sect: Sato Levine}. 

Our genus 3 results require slightly more technical language.  Suppose that $K$ is a genus 3 algebraically slice knot with derivative $L = L_1,L_2,L_3$ and Seifert matrix $M=\begin{bmatrix}0&A\\B&C\end{bmatrix}$.
The matrix $A$ induces a  linear operator on the alternating tensor $\Z^3\wedge \Z^3\cong \Z^3$ 
by the rule $(A\wedge A)(x\wedge y) = Ax \wedge Ay$.  Similarly, $A\wedge A$ acts on $(\Z/2)^3\wedge (\Z/2)^3\cong (\Z/2)^3$.

\begin{reptheorem}{main thm genus 3}
Let $K$ be an algebraically slice knot with genus 3 Seifert surface $F$ and derivative $L=L_1,L_2, L_3$.   Extend $L_1,L_2, L_3$ to a basis for $H_1(F)$.  Let $M=\begin{bmatrix}0&A\\B&C\end{bmatrix}$ be the resulting Seifert matrix.

If $\overline\mu_{123}(L)\in \Z$ is a multiple of $(\det(A)-\det(B))$ and $(A\wedge A-B^T\wedge B^T):(\Z/2)^3\wedge (\Z/2)^3\to(\Z/2)^3\wedge (\Z/2)^3$ is onto then $K$ is $1$-solvable.  
\end{reptheorem}

We recall the definition of Milnor's triple linking number   $\overline\mu_{123}(L)$ in Section~\ref{sect: triple linking}.

We can prove a similar theorem concluding the $1$-solvability of high genus knots.  In order to state it we gather together all of the triple linking numbers into a single invariant.  For a $g$-component link $L$, $\TL(L)\in \Z^g\wedge \Z^g\wedge\Z^g$ is defined by
$$
\TL(L) = \Sum_{1\le i<j<k\le g} \overline\mu_{123}(L_i,L_j, L_k)\cdot e_i\wedge e_j\wedge e_k,
$$
where $\{e_1,\dots, e_g\}$ is the standard basis for $\Z^g$ and $A^{\wedge 3}:\Z^g\wedge \Z^g\wedge \Z^g\to \Z^g\wedge \Z^g\wedge \Z^g$ is defined by the rule $A^{\wedge 3}(x\wedge y\wedge z) = Ax\wedge Ay\wedge Az$

\begin{reptheorem}{main thm genus g}
Let $K$ be an algebraically slice knot with genus $g$ Seifert surface $F$ and derivative $L=L_1, \dots, L_g$.   Extend $L_1,\dots, L_g$ to the basis for $H_1(F)$.  Let $M=\begin{bmatrix}0&A\\B&C\end{bmatrix}$ be the resulting Seifert matrix.  If $\TL(L)\in \Z^g\wedge\Z^g\wedge\Z^g$ is in the image of $A\wedge A \wedge A -B^T \wedge B^T \wedge B^T$ and $A\wedge A-B^T\wedge B^T:(\Z/2)^g\wedge (\Z/2)^g\to(\Z/2)^g\wedge (\Z/2)^g$ is onto, then $K$ is $1$-solvable.  
\end{reptheorem}

Theorem~\ref{main thm genus 3} follows as a corollary from Theorem~\ref{main thm genus g} by setting $g=3$.  

\begin{remark}\label{basis} \normalfont Once we fix a derivative $L=L_1, \dots, L_g$, there are infinitely many choices of $L_{g+1} \dots, L_{2g}$ so that $\{L_1,\dots, L_{2g}\}$ forms a basis for $H_1(F)$. This choice does not have an effect on our assumptions on Theorems~\ref{main thm genus 2},~\ref{main thm genus 3}, and~\ref{main thm genus g}. If we use a different basis to represent a Seifert matrix, we obtain $M'=\begin{bmatrix}0&AP\\P^TB&C\end{bmatrix}$ where $P$ is a change of basis matrix. Then we see that $\det(AP)-\det(P^TB) = \pm (\det(A)-\det(B))$. Further, $A\wedge A - B^T\wedge B^T$ and $AP\wedge AP - B^TP\wedge B^TP=(A\wedge A - B^T\wedge B^T)(P\wedge P)$ have the same image and $A\wedge A \wedge A- B^T\wedge B^T\wedge B^T$ and $(A\wedge A \wedge A - B^T\wedge B^T \wedge B^T)(P\wedge P \wedge P)$ also have the same image.
\end{remark}

\subsection{Applications}

In \cite{CHL10} Cochran-Harvey-Leidy produce knots spanning an infinitely generated subgroup of $\F_n/\F_{n.5}$.  Before we make an application of Theorem \ref{main thm genus 1} to their construction, we take a moment and recall it. They consider the operators $R_k$ of Figure~\ref{fig: CHL operator} and iterate these operators $n$ times. They prove that the union of their images contain an infinite linearly independent subset.

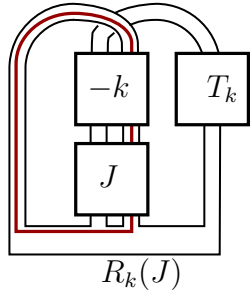
\begin{figure}[h]
\setlength{\unitlength}{1pt}
\begin{picture}(95,120)
\put(0,10){\includegraphics[height=.17\textheight]{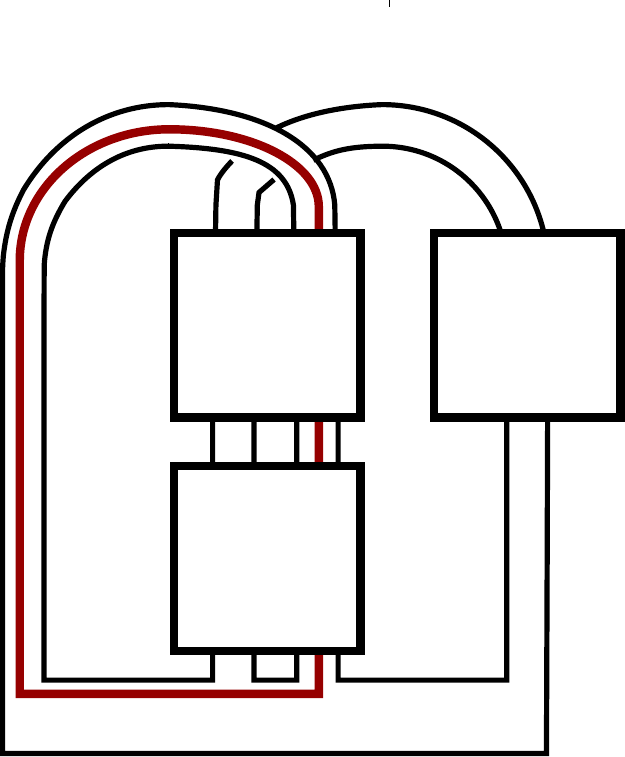}}
\put(34,36){$J$}
\put(30,70){$-k$} 
\put(75,70){$T_k$}
\put(35,0){$R_{k}(J)$}
\end{picture}
\caption{The Cochran-Harvey-Leidy knot $R_k(J)$ together with a derivative.  
The knot $T_k$ must be carefully chosen for the results of \cite{CHL10} to apply, but has no impact on our analysis.}\label{fig: CHL operator}
\end{figure}

\begin{proposition*}[\cite{CHL10}]
For $n\ge 1$, the set $$\left\{R_{k_n}(R_{k_{n-1}}({\dots}(R_{k_1}(J))\dots)) | k_1,\dots, k_n\in \Z, J\in \F_0\right\}$$
contains infinite set which is linearly independent in $\F_n/\F_{n.5}$.
\end{proposition*}

The purpose of the assumption $J\in \F_0$ is to guarantee that $S = R_{k_n}({\dots}(R_{k_1}(J))\dots)\in \F_n$.  Indeed, one might hope that if $J$ is not $0$-solvable, then $S$ will be $(n-0.5)$-solvable and not $n$-solvable.  As a Corollary of Theorem~\ref{main thm genus 1}, we prove that this is not the case.  The assumption that $J$ is $0$-solvable may be removed entirely from the examples of \cite{CHL10}.  

\begin{corollary}
For any $n\ge 1$, the knot $R_{k_n}({\dots}(R_{k_1}(J))\dots)$ is $n$-solvable.
\end{corollary}
\begin{proof}
The proof is a straightforward induction.  Notice first that $R_{k_1}(J)$ is algebraically slice and bounds a genus $1$ Seifert surface.  Theorem~\ref{main thm genus 1} now implies that $R_{k_1}(J)$ is $1$-solvable.  

Suppose that $R_{k_{n-1}}({\dots}(R_{k_1}(J))\dots)$ is $(n-1)$-solvable. $R_{k_n}(R_{k_{n-1}}({\dots}(R_{k_1}(J))\dots))$ has a genus 1 Seifert surface on which $R_{k_{n-1}}({\dots}(R_{k_1}(J))\dots)$ is a derivative (see Figure~\ref{fig: CHL operator}).  Theorem 8.9 of \cite{COT03} now concludes that $R_{k_n}(R_{k_{n-1}}({\dots}(R_{k_1}(J))\dots))$ is $n$-solvable. \end{proof}

As an additional application, we show that not only is every connected sum of genus 1 algebraically slice knots $1$-solvable, any knot which has the same Seifert matrix as a connected sum of algebraically slice genus 1 knots is $1$-solvable (see Figure~\ref{fig: onesolvablegenustwo}).  

\begin{corollary}
Let $K$ be a genus $2$ knot with Seifert form equal to $\begin{bmatrix}V&0\\0&W\end{bmatrix}$ where $V$ and $W$ are Seifert matrices for genus $1$ algebraically slice knots. Then $K$ is $1$-solvable.
\end{corollary}

\begin{proof}
For some choice of basis $V=\begin{bmatrix}0&x\\x-1&z\end{bmatrix}$ and $W=\begin{bmatrix}0&y\\y-1&w\end{bmatrix}$.  By changing basis we get $V\sim \begin{bmatrix}0&x-1\\x&z'\end{bmatrix}$ and $W\sim \begin{bmatrix}0&y-1\\y&w'\end{bmatrix}$ (see \cite[Corollary $4.6$]{Park16}). Thus, for an appropriate choice of basis we may assume that $x$ and $y$ have the same parity. After a permutation of basis elements, the Seifert matrix for $K$ becomes
$$\begin{bmatrix}
0&0&x&0\\
0&0&0&y\\
x-1&0& z&0\\
0&y-1&0&w\\
\end{bmatrix}.$$
At this point, Theorem~\ref{main thm genus 2} applies.\end{proof}

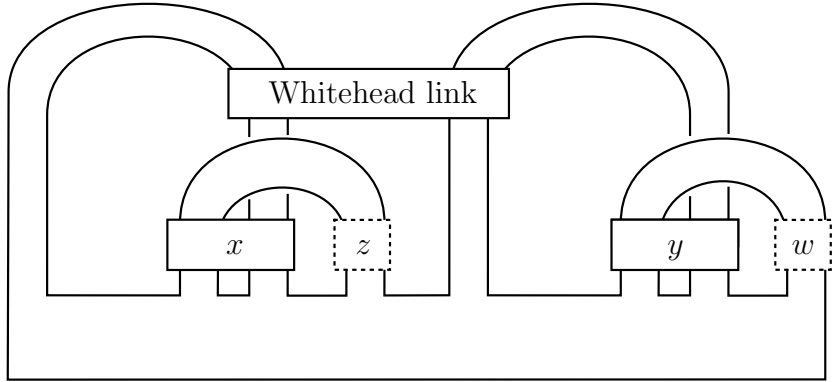
\begin{figure}[h]
\setlength{\unitlength}{1pt}
\begin{picture}(130,140)
\put(-80,0){\includegraphics[height=.22\textheight]{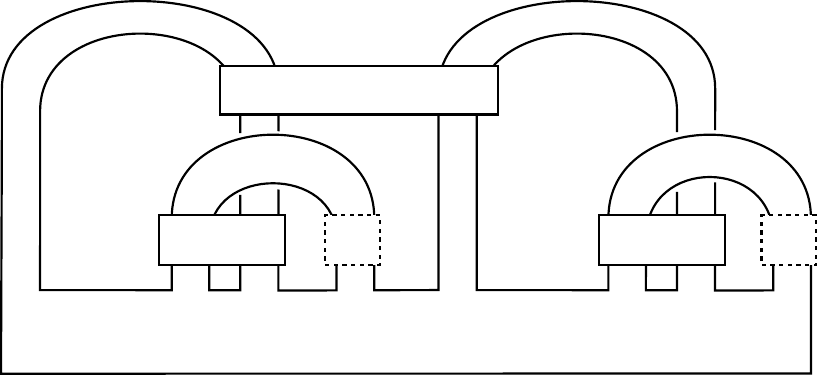}}
\put(52,48){$z$}
\put(03,48){$x$}
\put(217,48){$w$}
\put(170,48){$y$}
\put(19,105){Whitehead link}
\end{picture}
\caption{A family of $1$-solvable genus 2 knots. A solid box  with an integer represents full twists between two bands with no twist on each band. A dotted box with an integer represents full twists between two strands.}\label{fig: onesolvablegenustwo}
\end{figure}


The Alexander polynomial of a genus 2 algebraically slice knot is related to the difference of determinants which we assume is odd in Theorem~\ref{main thm genus 2}.  There is a sufficient condition for an algebraically slice genus 2 knot to be 1-solvable coming from the Alexander polynomial.  

\begin{corollary}\label{alexanderpoly}
Let $K$ be a genus $2$ algebraically slice knot with Alexander polynomial $\Delta_K(t) = a_4t^4+a_3t^3+ \cdots + a_1t+a_0$. If $a_4\equiv2 \mod 4$, then $K$ is $1$-solvable.
\end{corollary}
\begin{proof} Let $
M=\begin{bmatrix}0&A\\B&C\end{bmatrix}
$
be the Seifert matrix for $K$; then it is straightforward to verify that $\det(M) = -\det(A)\cdot \det(B) = a_4 \equiv 2 \mod 4$. Since $2$ divides $\det(A)\cdot \det(B)$ with multiplicity 1, exactly one of $\det(A)$ and $\det(B)$ is even.  Hence $\det(A)- \det(B)$ is an odd integer and  Theorem~\ref{main thm genus 2} concludes that $K$ is $1$-solvable. \end{proof}

\begin{remark} \normalfont$12a0596$ is a genus $2$ algebraically slice knot but not slice (see \cite{Knotinfo}). Since, $\Delta_{12a0596}(t) = 6t^4-20t^3+29t^2-20t+6$, we can apply Corollary~\ref{alexanderpoly} to conclude that $12a0596$ is $1$-solvable.
\end{remark}

Further, we show that there are many knots satisfying the assumption of Theorem~\ref{main thm genus 3} and we give a family of examples. 

\begin{corollary}\label{pm1}
Let $K$ be a genus $3$ algebraically slice knot with a Seifert matrix $M=\begin{bmatrix}
0 & A\\
B & C
 \end{bmatrix}$ such that $A= \begin{bmatrix}
x& 0 & 0 \\
0 & y & 0 \\
0 & 0 & z
 \end{bmatrix}$ and $B = A - \Id$. If $xyz-(x-1)(y-1)(z-1)= \pm 1$, then $K$ is $1$-solvable.
\end{corollary}
\begin{proof} Note that $\det(A) - \det(B) = \pm1$ and $A\wedge A-B^T\wedge B^T=\Id$ over $(\Z/2)^3\wedge (\Z/2)^3$. Then $K$ is $1$-solvable by Theorem~\ref{main thm genus 3}. \end{proof}

\begin{remark} \normalfont Note that for any integer $n$, if we let $x=n,y=-n,z=-n^2$, then $x,y,z$ satisfies $xyz-(x-1)(y-1)(z-1)=  1$. In particular, the family of knots depicted in Figure~\ref{fig: onesolvablegenusthree} is $1$-solvable by Corollary~\ref{pm1}.
\end{remark}

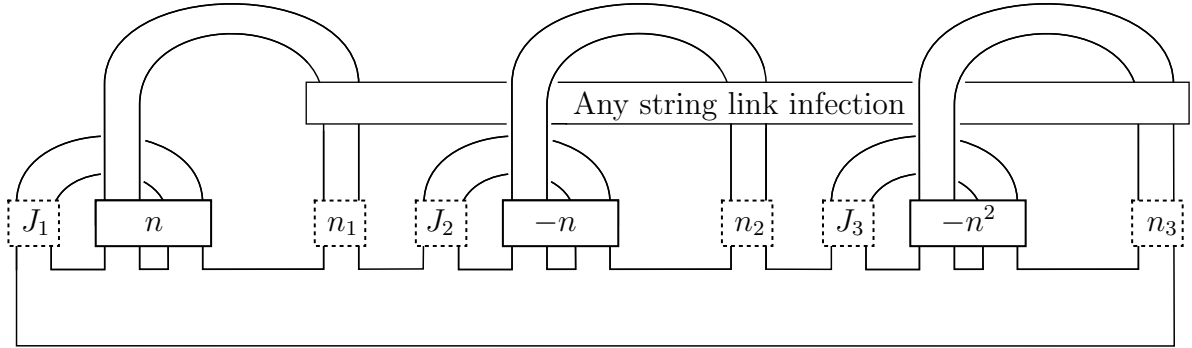
\begin{figure}[h]
\setlength{\unitlength}{1pt}
\begin{picture}(130,140)
\put(-158,0){\includegraphics[height=.20\textheight]{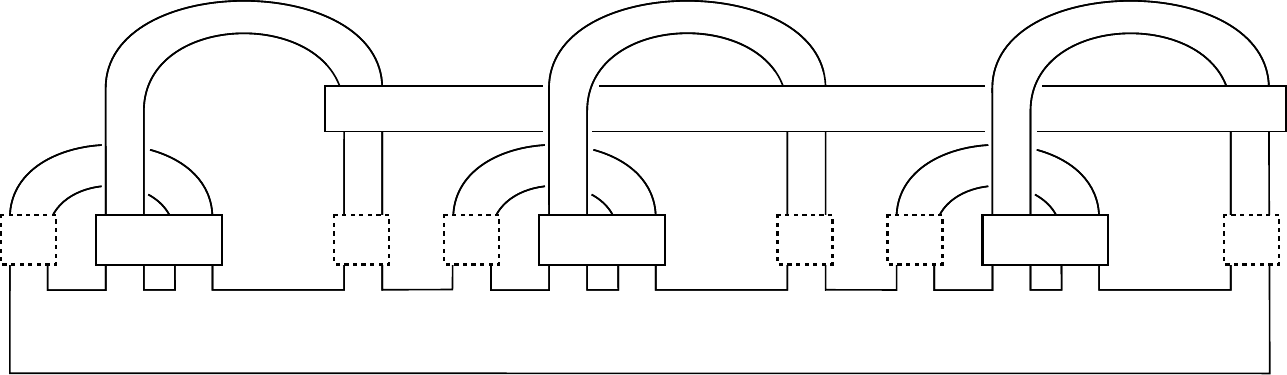}}
\put(-106,44){$n$}
\put(-37,44){$n_1$}
\put(117,44){$n_2$}
\put(273,44){$n_3$}
\put(41,44){$-n$}
\put(56,88){Any string link infection}
\put(-153,44){$J_1$}
\put(0,44){$J_2$}
\put(155,44){$J_3$}
\put(195,44){$-n^2$}
\end{picture}
\caption{A family of $1$-solvable genus $3$ knots. A solid box with an integer represents full twists between two bands with no twist on each band. A dotted box with an integer represents full twists between two strands. A dotted box with the knot $J_i$ represents a band tied into $J_i$. }\label{fig: onesolvablegenusthree}
\end{figure}

\subsection{Knots in homology spheres, homology concordance, and the solvable filtration.}\label{sect:filtration}

Before we proceed, we recall precisely what it means for a $k$-component link to be $n$-solvable.  To specialize to knots, simply take $k=1$.    For a link $L$, $M(L)$ denotes the 3-manifold given by $0$-framed surgery along each component of $L$.  For a group $G$, the derived series, $G^{(n)}$, is defined recursively by $G^{(0)}=G$ and $G^{(n+1)} = [G^{(n)}, G^{(n)}]$.

\begin{definition}
For $k, n\in \N$, a $k$-component link $L$ with zero pairwise linking numbers is called \textbf{$n$-solvable} if there exists a smooth compact oriented $4$-manifold $W$ (called an $n$-solution) with $\bdry W=M(L)$ such that 
\begin{enumerate}
\item $\Z^k\cong H_1(M(L))\to H_1(W)$ is an isomorphism.
\item $H_2(W)\cong \Z^{2\ell}$ has a basis $\{L_1,\dots, L_\ell, D_1,\dots, D_\ell\}$ consisting of smoothly embedded surfaces with trivial normal bundles which are all disjoint except that $L_i$ intersects $D_i$ transversely in a single point.
\item The inclusion induced maps $\pi_1(L_i)\to \pi_1(W)$ and $\pi_1(D_i)\to \pi_1(W)$ have images contained in $\pi_1(W)^{(n)}$.
\end{enumerate}
If additionally, 
$\im(\pi_1(L_i)\to \pi_1(W))\subseteq \pi_1(W)^{(n+1)}$
 then $L$ is called $n.5$-solvable and $W$ is called an $n.5$-solution.  
\end{definition}

Note that this definition of $n$-solvability makes no reference to the link $L$ being in $S^3$.  The concept of solvability makes perfectly good sense for links in any homology sphere.  This is important because many of our techniques will require us to consider links in homology spheres other than $S^3$ and to consider concordances in homology cobordisms.  If $L=L_1,\dots,L_k\subseteq M$ and $J=J_1,\dots, J_k\subseteq N$ are links in homology spheres then we will say that $L$ is \textbf{homology concordant} to $J$ if there exists a homology cobordism $V$ from $M$ to $N$ in which there exist disjoint embedded annuli $A=A_1,\dots,A_k$ such that for $i=1,\dots,k$, $L_i$ and $J_i$ cobound $A_i$.  A is called a \textbf{homology concordance}.

Suppose that $J$ is $n$-solvable and $L$ and $J$ are homology concordant.  It follows that their $0$-framed surgeries are homology cobordant.  By gluing this homology cobordism to an $n$-solution for $J$, one sees an $n$-solution for $L$.  Thus, the $n$-solvability of a link depends only on its homology concordance class.  

\subsection*{Acknowledgements}
The authors thank Kent Orr for several helpful conversations. We also thank Mark Powell for suggesting collaboration between the authors. Finally, we are grateful to the anonymous referee for thoughtful suggestions.

\section{Generalized satellite operators}

In \cite{CD15}, Cochran and the first author give a satellite operator based construction which preserves the concordance class and genus of a knot but changes its derivatives.  In \cite{CD16} appears a generalization of this construction, which we recall in this section.  It has the property that it sends knots in $S^3$ to knots in other homology spheres.  In Proposition~\ref{modification lemma}, we give a means using this generalized satellite operation to preserve the concordance class of a knot while changing its derivatives.

\begin{definition}[See also Definition 5.2 of \cite{CD16}]
Let $P\subseteq M$ be a knot in a homology sphere, and $\eta\subseteq M-P$ be a knot disjoint from $P$ with a chosen framing $q$.  Let $J\subseteq N$ be a knot in a homology sphere.  Let $n(\eta)$ and $n(J)$ be open tubular neighborhoods of $\eta$ and $J$ respectively.  One can build a new homology sphere $M_{\eta}(J)$ by gluing together the complements $M-n(\eta)$ and $N-n(J)$ so that the meridian of $\eta$ is identified with the $0$-framed longitude of $J$ and the $q$-framed longitude of $\eta$ is identified with the meridian of $J$. \textbf{The generalized satellite knot} $P_\eta(J)\subseteq M_\eta(J)$ is the image of $P$ in this construction.  
\end{definition}

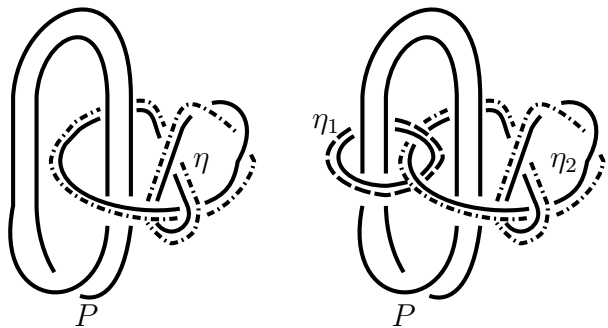
\begin{figure}[h]
\setlength{\unitlength}{1pt}
\begin{picture}(220,120)
\put(0,10){\includegraphics[height=.17\textheight]{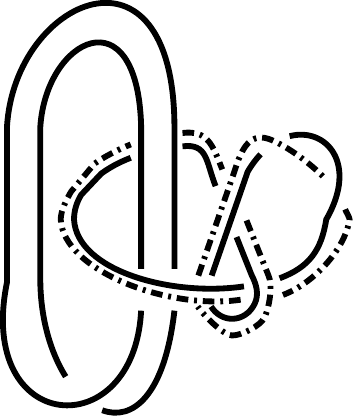}}
\put(25,0){$P$}
\put(70,60){$\eta$}
\put(120,10){\includegraphics[height=.17\textheight]{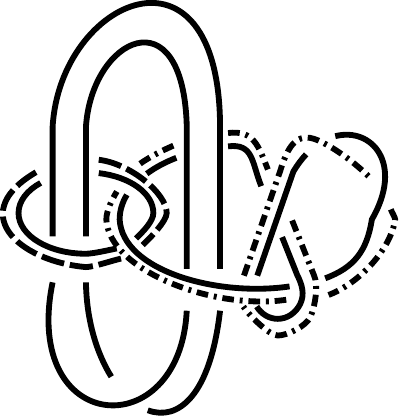}}
\put(145,0){$P$}
\put(115,75){$\eta_1$}
\put(205,60){$\eta_2$}
\end{picture}
\caption{Left: A generalized satellite operator $P_\eta$.  Right:  The operator $P_{\eta_1,\eta_2}$.  Notice that the curves $\eta$, $\eta_1$, and $\eta_2$ are allowed to have non-trivial knotting, linking, and framing.  The framings are indicated above by a choice of pushoff.}\label{fig:GeneralizedSatellite}
\end{figure}

The pair $(P, \eta)$, or $P_\eta$, is called a \textbf{generalized satellite operator}. See Figure~\ref{fig:GeneralizedSatellite}.  This definition differs somewhat from the definition presented in \cite{DR}.  When $\eta$ is a $0$-framed unknot this recovers the classical satellite construction.  We take a moment and justify the assertion that $M_\eta(J)$ is a homology sphere.  Notice the meridians of $\eta$ and $J$ respectively generate $H_1(M-n(\eta))$ and $H_1(N-n(J))$ while the $0$-framed longitude of $J$ is nullhomologus.  A Mayer-Veitoris argument shows that $M_\eta(J)$ is a homology sphere.  

Suppose $P$, $\eta_1$, and $\eta_2$ are disjoint curves in the homology sphere $M$,  $\eta_1$ and $\eta_2$ are framed, and $J_1\subseteq N_1$, $J_2\subseteq N_2$ are knots in homology spheres. By performing the satellite operation first along $\eta_1$ and then along $\eta_2$, one forms the knot $P_{\eta_1,\eta_2}(J_1,J_2)$.   

\begin{proposition}[\cite{CD15}]\label{Arf and infection}
Let $P_\eta$ be a generalized satellite operator and $J\subseteq M$ be a knot in a homology sphere.  Let $w=\lnk(P,\eta)$, then $\Arf(P_\eta(J)) = \Arf(P)+w\Arf(J)$.  
\end{proposition}

Before we can prove Theorem~\ref{main thm genus 1}, we need one last result from \cite{CD16}.  

\begin{proposition}[Proposition 5.3 of \cite{CD16}.  See also Corollary 2.3 of \cite{CD15}.]\label{modification lemma}
Let $P\subseteq M$ be a knot in a homology sphere.  Let  $\eta_1$ and $\eta_2$ be disjoint knots in $M-n(P)$ which cobound an annulus $A$ in $M-n(P)$.  Pick framings on these knots so that their framed longitudes cobound a pushoff of $A$.  Then, $P_{\eta_1,\eta_2}(J,-J)$ is homology concordant to $P$ for any knot $J$.  
\end{proposition}

We are now ready to prove the first of the main theorems.

\section{Proof of Theorem \ref{main thm genus 1}.}\label{sect: proof of 1}

In this section we prove the following result, from which Theorem \ref{main thm genus 1} follows.

\begin{proposition}\label{main thm genus 1'}
Let $K$ be an algebraically slice knot of genus $g$.  Then $K$ is homology concordant to a knot $K'$ bounding a genus $g$ Seifert surface with a derivative for which every component is $0$-solvable as a knot.  
\end{proposition}

\begin{proof}

Let $K\subseteq M$ be an algebraically slice knot in a homology sphere.  Let $F$ be a genus $g$ Seifert surface with derivative $L=L_1,\dots, L_g$.  If $\Arf(L_i)=0$ for all $i$ then we are done.  Otherwise, we reorder $L$ so that  $\Arf(L_{1}) = \dots = \Arf(L_{k}) = 0$ and $\Arf(L_{k+1}) = \dots = \Arf(L_{g})= 1$.  There exists a curve $\eta$ on $F$ disjoint from $L_1, \dots, L_k$ which intersects each of $L_{k+1},\dots, L_g$ transversely in a single point.  

Notice that the pushoffs $\eta^+$ and $\eta^-$ cobound an annulus in the complement of $K$.  For $q=\lnk(\eta^+,\eta^-)$, the $q$-framed longitudes cobound a pushoff of this annulus. Proposition~\ref{modification lemma} applies and $K':=K_{\eta^+,\eta^-}(J,-J)$ is homology concordant to $K$ for any knot $J$.  

Let $J$ be a knot with $\Arf(J)= 1$.  Notice that since $\eta^+$ and $\eta^-$ are disjoint from the genus $g$ surface $F$, $K'$ still bounds a genus $g$ Seifert surface.  Moreover, on this Seifert surface we see derivative  $L' := L_{\eta^+,\eta^-}(J,-J)$.  By Proposition~\ref{Arf and infection}, 
$$
\Arf(L_i') = \Arf(L_i) + \lnk(L_i, \eta^+) - \lnk(L_i, \eta^-).
$$
Now, $\lnk(L_i, \eta^+) - \lnk(L_i, \eta^-)$ is precisely the algebraic intersection number $L_i\cdot \eta$.   Thus, for $i=1, \dots, k$, 
$$
\Arf(L_i') = \Arf(L_i)+L_i\cdot \eta= 0+0=0,
$$
while for $i=k+1,\dots,g$
$$
\Arf(L_i') = \Arf(L_i)+L_i\cdot \eta= 1+1= 0.
$$

Thus, $K$ is homology concordant to $K'$ which admits a genus $g$ Seifert surface with a derivative $L'$, where each component is $0$-solvable. \end{proof}

We are ready to prove Theorem~\ref{main thm genus 1}.

\begin{theorem}\label{main thm genus 1}
If a knot $K$ has genus $1$ and is algebraically slice, then $K$ is $1$-solvable.
\end{theorem}

\begin{proof}
Let $K$ be a genus $1$ algebraically slice knot.  According to Proposition \ref{main thm genus 1'}, $K$ is homology concordant to a knot $K'$ in a homology sphere which admits a genus $1$ Seifert surface on which there lies a derivative $L'$ whose every component is $0$-solvable.  Since $L'$ is a derivative on a genus $1$ Seifert surface, $L'$ has only one component.  Thus, $L'$ is $0$-solvable and by \cite[Proposition 8.9]{COT03}, $K'$ is $1$-solvable.  Since $K$ is homology concordant to $K'$, $K$ is $1$-solvable. \end{proof}

\section{A string link modification lemma.}\label{sect:modification lemma}

The techniques of the previous section fail for higher genus knots, as the Arf invariant is insufficient to classify $0$-solvability of links. Rather, in \cite{Martin15}, the second author provides a complete set of obstructions to $0$-solvability.  
Namely, they are pairwise linking numbers, the Arf-invariant of the individual components, the Sato-Levine invariants (mod 2) of the $2$-component sublinks and the triple linking numbers of the $3$-component sublinks.  We will define these invariants when we need them in Sections \ref{sect: Sato Levine} and \ref{sect: triple linking}.

The satellite operation (even the generalized satellite operation)  changes the $\Arf$-invariant, but cannot change the Sato-Levine invariant or triple linking number.  If we are to replace a derivative with a $0$-solvable link, we must provide a new modification.  

\begin{definition}
Let $D$ be the closed 2-dimensional unit disk.  Let $p_1,\dots,p_n$ be $n$ distinct points in the interior of $D$.  An \textbf{$n$-component pure string link} is a collection of $n$ disjoint properly embedded arcs $L_1,\dots, L_n$ in $D\times[0,1]$  such that $L_k$ runs from $(p_k,0)$ to $(p_k,1)$.
\end{definition}

Since every string link considered in this paper is a pure string link we will suppress the word \emph{pure} from our notation throughout.
Given a string link $L$ we let $-L$ denote the reverse of the mirror image of $L$.  

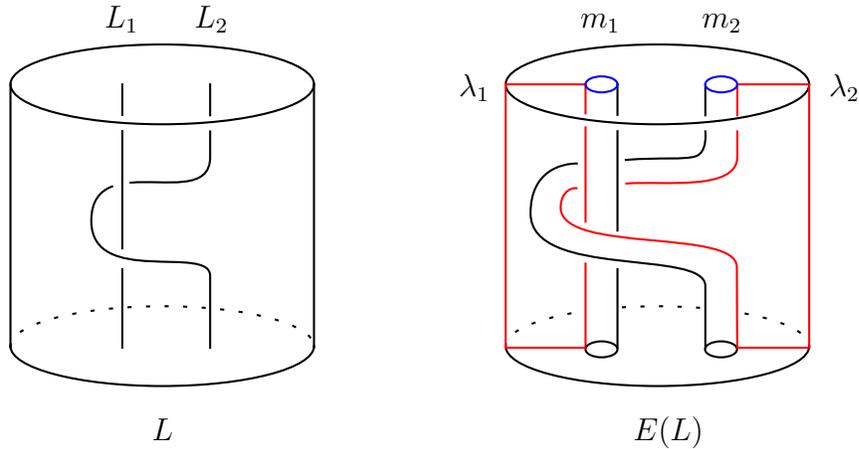
\begin{figure}[h]
\setlength{\unitlength}{1pt}
\begin{picture}(0,170)
\put(-150,20){\includegraphics[height=.20\textheight]{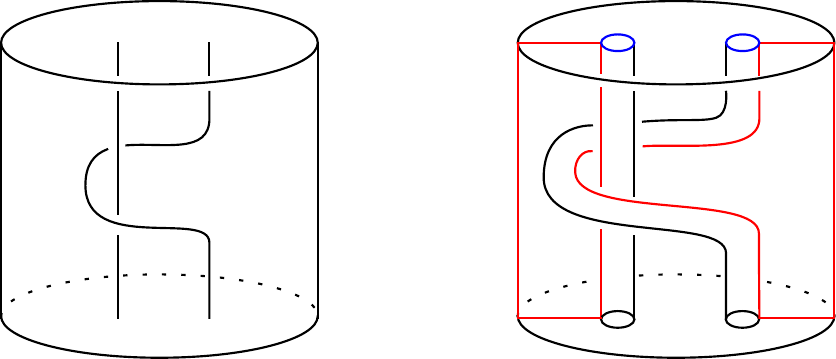}}
\put(-96,0){$L$}
\put(-114,156){$L_1$}
\put(-80,156){$L_2$}
\put(86,0){$E(L)$}
\put(66,156){$m_1$}
\put(112,156){$m_2$}
\put(20,130){$\lambda_1$}
\put(160,130){$\lambda_2$}
\end{picture}
\caption{A string link $L$ and its complement $E(L)$. The blue circles $m_1$ and $m_2$ are meridians and the red circles $\lambda_1$ and $\lambda_2$ are $0$-framed longitudes.}\label{fig: stringlink}
\end{figure}

For any string link $L = L_1,\dots, L_n$ let $E(L)$ be the complement of a neighborhood of $L$.  The meridian of $L_k$, $m_k\subseteq \bdry E(L)$ is the boundary of a small disk centered at $(p_k,0)$.  The $0$-framed longitude $\lambda_k\subseteq \bdry E(L)$ of $L_k$ intersects $m_k$ transversely in a single point, is disjoint from all other meridians, and is nullhomologous in $E(L_k)$, the exterior of the arc $L_k$.   See Figure~\ref{fig: stringlink}.  When $L$ is not clear from context, we specify $m_k(L)$ and $\lambda_k(L)$.   The linking number between $L_j$ and $L_k$ is given by counting up intersection points (with sign) between $L_j$ and a surface bounded by $\lambda_k$ in $E(L_k)$.  Thus, $L$ has zero pairwise linking numbers if and only if each $\lambda_k$ is nullhomologous in the complement of $L$.  For $q\in \Z$, the $q$-framed longitude of $L_k$ consists of the $0$-framed longitude together with $p$ copies of the meridian of $L_k$.  

Let $P\subseteq M$ be a link in a homology sphere.  Let $\alpha = \alpha_1\vee\dots\vee\alpha_n$ be an embedded wedge of circles in the complement of $P$.  Let $m_1,\dots, m_n$ be a collection of meridians for $\alpha_1,\dots,\alpha_n$.  Let $\lambda_1,\dots, \lambda_n$ be a disjoint collection of (not necessarily nullhomologous) longitudes.   When $\alpha$ is not clear from context we specify $m_k(\alpha)$ and $\lambda_k(\alpha)$.  The pair $(P,\alpha)$ or $P_\alpha$ is a \textbf{generalized string link operator}. See Figure \ref{fig: generalized operator}. $P_\alpha$ is a \textbf{classical string link operator} if the longitudes $\lambda_1,\dots,\lambda_n$ bound disjoint disks in the complement of $\alpha$.  

\begin{figure}[h]
\setlength{\unitlength}{1pt}
\begin{picture}(0,120)
\put(-58,0){\includegraphics[height=.17\textheight]{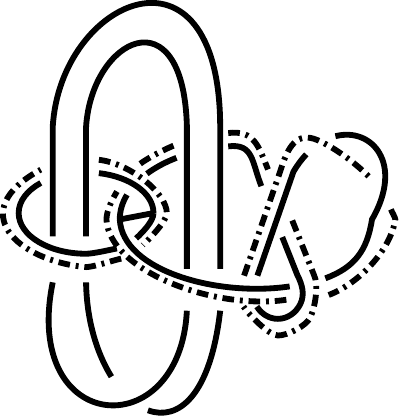}}
\put(50,33){$\alpha$}
\end{picture}
\caption{A generalized string link operator $P_\alpha$.  The dashed lines indicate a choice of longitudes. 
}\label{fig: generalized operator}
\end{figure}

Given a string link $L$ with zero linking numbers and a generalized string link operator $P_\alpha$ in a homology sphere $M$, one can form a new homology sphere $M_\alpha(L)$ by cutting out a neighborhood of $\alpha$, $n(\alpha)$ and gluing in the complement of $L$ so that $m_k(\alpha)$ is identified with $\lambda_k(L)$ and so that $m_k(L)$ is identified with $\lambda_k(\alpha)$.

As is the case with generalized satellite operators, $M_\alpha(L)$  is a homology sphere.  The result of infection, $P_\alpha(L)\subseteq M_\alpha(L)$ is the image of $P$ in this new manifold.   Let $\alpha$ and $\beta$ be disjoint wedges of circles with choices of longitudes. Let $L$ and $ J$ be string links with the appropriate number of components. We iteratively form $P_{\alpha, \beta}(L,J)$. We introduce the following definition to be suitable to our needs. 

\begin{definition}\label{abstract wedge}
Let $W$ be a homology cobordism. Let $\alpha$ and $\beta$ be disjoint wedges of $k$-circles in a single boundary component of $W$. We say $\alpha$ and $\beta$ are {\bf framed wedge cobordant} if, for a genus $k$ handlebody $H_k$ with meridonal curves $m_1, \dots, m_k$ and longitudinal curves $\lambda_1, \dots, \lambda_k$, there exists a smooth embedding $\Phi: H_k \times [0,1] \to W$ such that: 
\begin{itemize}
\item the restriction $\Phi|_{H_k\times\{0\}}$ sends the meridonal curves for $H_k$ to $m(\alpha_i)$ ,
\item $\Phi|_{H_k\times\{0\}}$ sends longitudes to $\lambda(\alpha_i)$,
\item  $\Phi|_{H_k\times\{1\}}$ sends  meridians to the reverse of $m(\beta_i)$, and
\item $\Phi|_{H_k\times\{1\}}$ sends  longitudes to $\lambda(\beta_i)$. 
\end{itemize}
\end{definition}

In \cite[Lemma 3.2]{JKP} Jang, Kim and Powell generalize \cite[Theorem 3.1]{CD15} to the setting of classical string link infection.  The result further generalizes to generalized string link infection.  Our proof is highly similar to theirs, as well as \cite[Theorems 2.2 and 3.1]{Park15},  and we merely summarize the ideas.

\begin{proposition}\label{string link mod}
Let $P\subseteq M$ be a link in a homology sphere. Let $\alpha, \beta \subseteq M$ be framed wedge cobordant in $(M - P) \times [0,1]$. Then for any $n$-component string link $L$, 
$
P_{\alpha,\beta}(L,-L)
$
is homology concordant to $P$.
\end{proposition}

\begin{proof}

Let $A = \Phi( H_k\times[0,1])$. Construct a 4-manifold $W$ by starting with $M\times[0,1]$ and replacing $A$ with $E(L)\times[0,1]$ so that $\Phi(m(L_k)\times [0,1])$ is identified to $m(\alpha_k) \times [0,1]$ and $\Phi(\lambda(L_k)\times [0,1])$ is identified with $\lambda(\alpha_k) \times [0,1]$.  This modification does not change $P\times\{0\}\subseteq M\times\{0\}$ and replaces $P\times\{1\}\subseteq M\times\{1\}$ by $P_{\alpha,\beta}(L,-L)$.  Since the annulus $P\times[0,1]$ is disjoint from $A$, we see an embedded annulus in $W$ bounded by $P$ and $P_{\alpha,\beta}(L,-L)$.  

It remains only to check that $W$ is a homology cobordism. A Mayer-Vietoris argument reveals that $H_0(M\times[0,1]-A)\cong H_3(M\times[0,1]-A)\cong \Z$, $H_1(M\times[0,1]-A)\cong \Z^k$ is generated by the meridians $m(\alpha_1),\dots,m(\alpha_k)$ and all other homology groups vanish.  In $W$, the meridians are identified to the nullhomologous longitudes of $L_1,\dots,L_k$.  From here another Mayer-Vietoris reveals that $H_0(W)\cong H_3(W)\cong \Z$ while the other homology groups vanish.  Thus, $W$ is a homology cobordism in which $P\times[0,1]$ is a homology concordance. \end{proof}

\section{String link infection and the Sato-Levine invariant:  The proof of Theorem~\ref{main thm genus 2}}\label{sect: Sato Levine}


There is an invariant of $2$-component links with vanishing pairwise linking number called the \textbf{Sato-Levine invariant} \cite{Sato84}.  As in \cite{Cochran85, Cochran90} this invariant agrees with Milnor's invariant $\overline\mu_{iijj}(L)$.  

\begin{definition}
Let $L=L_1,L_2$ be a $2$-component link.  There exist Seifert surfaces $F_1$ and $F_2$ for $L_1$ and $L_2$ which intersect transversely  in a collection of disjoint oriented simple closed curves $a_1,\dots,a_n$.  Let $a_j^+$ be the result of pushing $a_j$ in the positive normal direction off of $F_1$.  The Sato-Levine invariant of $L$ is given by 
$$
\overline\mu_{1122}(L) = \Sum_{1\le i,j \le n}\lnk(a_i, a_j^+).
$$
By modifying $F_1$ and $F_2$, we can assume that $F_1\cap F_2$ consists of a single simple closed curve, $a$.  In this case $  \overline\mu_{1122}(L) = \lnk(a, a^+)$.
\end{definition}

In \cite{Martin15} the second author completely determines when a link is $0$-solvable.  Specializing this result to the setting of $2$-component links gives the following result
\begin{proposition*}[Theorem 1 of \cite{Martin15}]
Let $L = L_1,L_2$ be a $2$-component link.  Then $L$ is $0$-solvable if and only if 
\begin{enumerate}
\item $\lnk(L_1,L_2) = 0$, 
\item $\Arf(L_1) = \Arf(L_2) = 0$ and 
\item $\overline{\mu}_{1122}(L) \equiv 0 \mod2$.
\end{enumerate}
\end{proposition*}

\begin{corollary}\label{genus 2 step 1}
Suppose that $K$ is algebraically slice and bounds a genus $2$ Seifert surface $F$ on which lies a derivative $L=L_1,L_1$.  
  If $\overline{\mu}_{1122}(L) \equiv 0 \mod 2$, $K$ is $1$-solvable.
\end{corollary}
\begin{proof}
As in Proposition~\ref{main thm genus 1'}, we see that there is a knot $K'$ which is homology concordant to $K$ with derivative $L'=L_1',L_2'$ such that $\Arf(L_1')=\Arf(L_2')=0$.  The derivative $L'$ is given by modifying $L$ by a generalized satellite operation.  It is straightforward to see (by studying bounded Seifert surfaces) that this does not change the Sato-Levine invariant.  Thus, $\overline \mu_{1122}(L') = \overline\mu_{1122}(L)\equiv 0 \mod 2$.  

By \cite[Theorem 1]{Martin15}, $L'$ is $0$-solvable, so that by \cite[Propositon 8.9]{COT03} $K'$ is $1$-solvable.  Since $K$ is homology concordant to $K'$, $K$ is $1$-solvable. \end{proof}


In order to complete the proof of Theorem~\ref{main thm genus 2} we only need to find a means to replace the knot $K$ with a new genus $2$ knot $K'$ admitting a derivative $L'$ with even Sato-Levine invariant.  With this goal in mind we determine how generalized string link infection affects the Sato-Levine invariant. 

For any string link $J$, $\widehat{J}$ is the result of closing up $J$ to a link in $S^3$.  See Figure~\ref{fig: closure}.

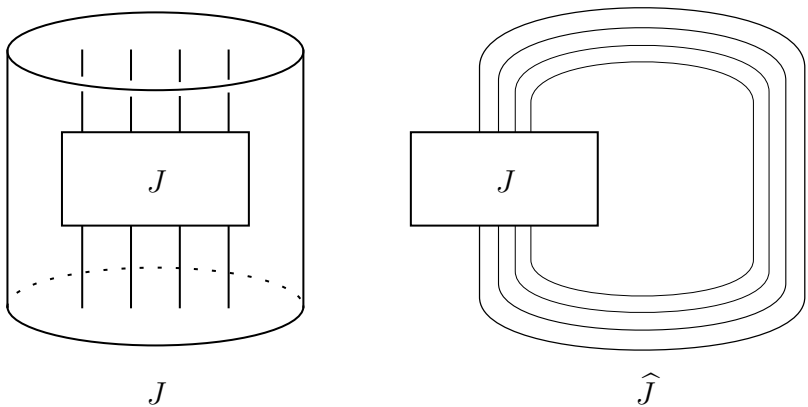
\begin{figure}[h]
\setlength{\unitlength}{1pt}
\begin{picture}(0,160)
\put(-150,20){\includegraphics[height=.20\textheight]{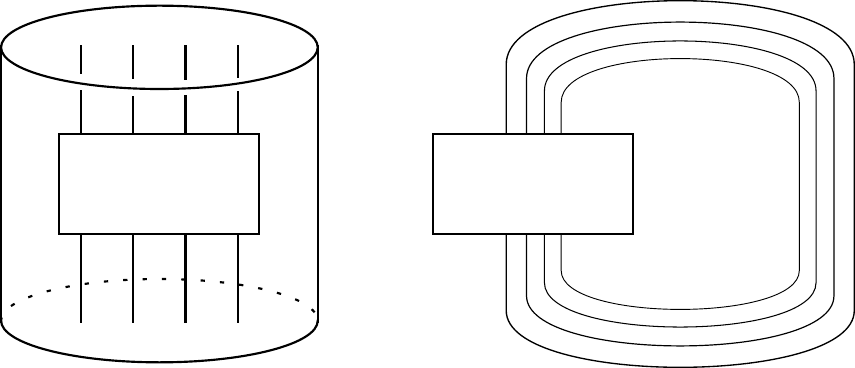}}
\put(-97,0){$J$}
\put(-97,80){$J$}
\put(88,0){$\widehat{J}$}
\put(35,80){$J$}
\end{picture}
\caption{A string link $J$ and its closure $\widehat{J}$ in $S^3$.}\label{fig: closure}
\end{figure}

\begin{proposition}\label{prop:infection and SL}
Let $L = L_1, L_2$ be a $2$-component link.  Let $\alpha = \alpha_1\vee \alpha_2$ be a wedge of $2$ circles in the complement of $L$.  Let $J$ be a $2$-component string link with zero linking number.  Let $C = (c_{i,j})$ be the $2\times 2$ matrix given by $c_{i,j} = \lnk(L_i, \alpha_j)$.  Then 
$$
\begin{array}{rcl}\overline\mu_{1122}(L_\alpha(J)) &=& \overline\mu_{1122}(L) + (c_{1,1}c_{2,2}-c_{1,2}c_{2,1})^2\overline\mu_{1122}(\widehat{J}) 
\\&=& \overline\mu_{1122}(L) + \det(C)^2\overline\mu_{1122}(\widehat{J})\end{array}$$
\end{proposition}

For any integer $x$, $x^2\equiv x \mod 2$ so we get the following immediate consequence of Proposition \ref{prop:infection and SL}.  

\begin{corollary}\label{cor: infection and SL}
Let $L = L_1, L_2$ be a $2$-component link.  Let $\alpha = \alpha_1\vee \alpha_2$ be a wedge of two circles in the complement of $L$.  Let $J$ be a $2$-component string link with zero linking number.  Let $C = (c_{i,j})$ be the $2\times 2$ matrix given by $c_{i,j} = \lnk(L_i, \alpha_j)$.  Then 
$$
\overline\mu_{1122}(L_\alpha(J)) \equiv \overline\mu_{1122}(L) + \det(C)\overline\mu_{1122}(\widehat{J})  \mod  2.$$
\end{corollary}

\begin{proof}[Proof of Proposition~\ref{prop:infection and SL}]
Let $F_1$ and $F_2$ be Seifert surfaces for $L_1$ and $L_2$ respectively such that $F_i$ intersects $\alpha_j$ transversely in $c_{i,j}$ points with positive orientation. (Or $-c_{i,j}$ points with negative orientation if $c_{i,j}<0$.) Isotoping slightly if needed we can assume that these points of intersection are disjoint from the intersection curve $a=F_1\cap F_2$.  

 Let $E_1$ and $E_2$ be surfaces bounded by the longitudes of $J_1$ and $J_2$ respectively.   We assume that $E_1\cap E_2$ is a simple closed curve.   Let $d$ be this curve so $\lnk(d,d^+) = \overline{\mu}_{1122}(\widehat{J})$.

Let $i=1$ or $2$.  In order to build a Seifert surface for the $i$'th component of $L_\alpha(J)$, start with $F_i$, cut out the $c_{i,1}$ points of intersection with $\alpha_1$ and glue in $c_{i,1}$ parallel copies of $E_1$.  Go on to cut out the $c_{i,2}$ points of intersection with $\alpha_2$ and glue in $c_{i,2}$ parallel copies of $E_2$.  Since $E_1$ and $E_2$ intersect in the curve $d$ we see that this new surface is no longer embedded, but crosses itself  in $c_{i,1} c_{i,2}$ parallel copies of $d$.  We can smooth these intersection curves as in Figure~\ref{fig: smoothing} to get a Seifert surface $F_i'$.

\begin{figure}[h]
\setlength{\unitlength}{1pt}
\begin{picture}(0,145)
\put(-180,10){\includegraphics[height=.20\textheight]{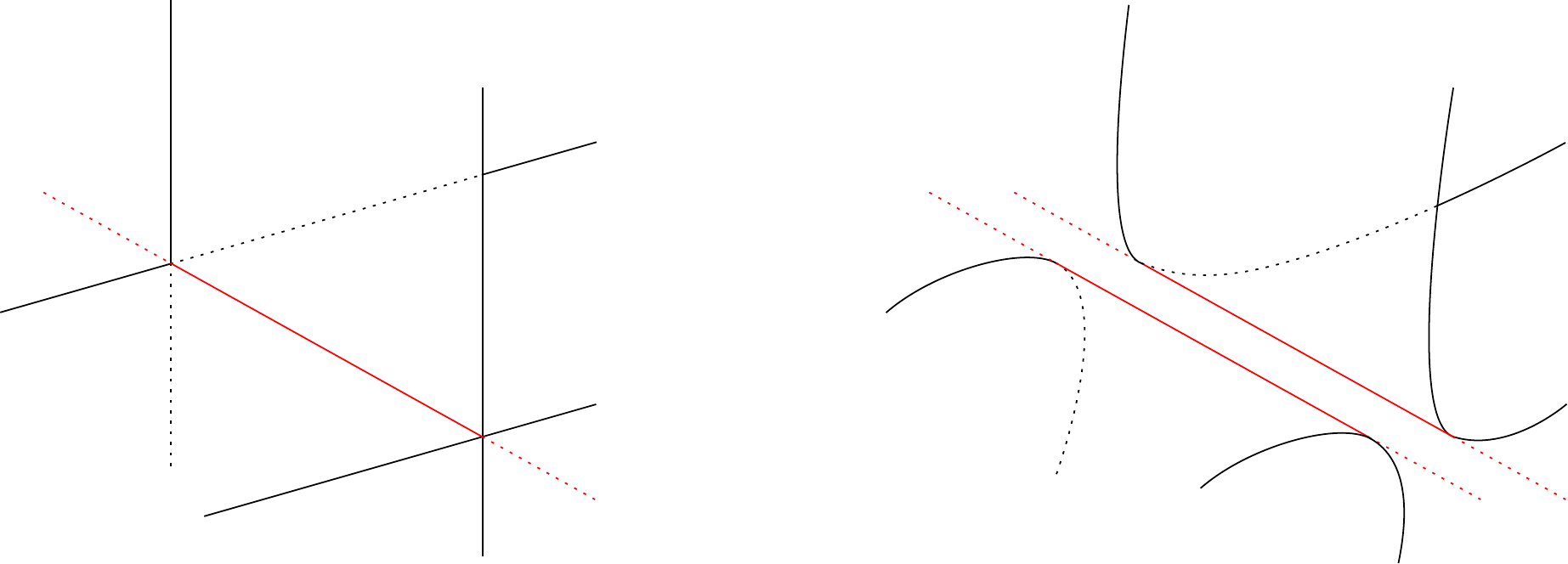}}
\put(-16,65){$\longrightarrow$}
\end{picture}
\caption{Smoothing self intersection curves to get a Seifert surface.}\label{fig: smoothing}
\end{figure}

The intersection $F_1'\cap F_2'$ consists of the original intersection curve $a$, $k = c_{1,1}c_{2,2}$  parallel copies of $d$ (call them $d_1,\dots,d_k$), and $\ell = c_{1,2}c_{2,1}$ parallel copies of $d$ with their orientation reversed  ($e_1,\dots,e_\ell$).  Since each of these parallel copies come from pushing $d$ off itself using the $E_1$ normal direction we see that $\lnk(d_i,d_j^+)=\lnk(e_i,e_j^+)=\lnk(d,d^+) = -\lnk(d_i,e_j^+) = \overline{\mu}_{1122}(\widehat{J})$.  
Adding these up we see
$$
\begin{array}{rcl}
\overline\mu_{1122}(L_\alpha(J)) &=& \overline\mu_{1122}(L) + (k^2+\ell^2 - 2k\ell)\overline\mu_{1122}(\widehat{J})
\\&=& \overline\mu_{1122}(L) + (k-\ell)^2\overline\mu_{1122}(\widehat{J})
\\&=& \overline\mu_{1122}(L) + (c_{1,1}c_{2,2}-c_{1,2}c_{2,1})^2\overline\mu_{1122}(\widehat{J})
\end{array}
$$
as we claimed. \end{proof}

Next, we prove the following proposition from which Theorem~\ref{main thm genus 2} follows.

\begin{proposition}\label{Change iijj}
Let $K$ be an algebraically slice knot with genus $2$ Seifert surface $F$ and derivative $L=L_1,L_2$.   Suppose that there exist disjoint curves $\alpha_1, \alpha_2$ on $F$  such that
$$
\det\begin{bmatrix}\lnk(L_1,\alpha_1^+) &\lnk(L_1,\alpha_2^+)\\
\lnk(L_2,\alpha_1^+) &\lnk(L_2,\alpha_2^+)\end{bmatrix}
-
\det\begin{bmatrix}\lnk(L_1,\alpha_1^-) &\lnk(L_1,\alpha_2^-)\\\lnk(L_2,\alpha_1^-) &\lnk(L_2,\alpha_2^-)\end{bmatrix}
\equiv
1 \mod 2.
$$
Then $K$ is $1$-solvable.  
\end{proposition}
\begin{proof}

Isotoping $\alpha_1$ until it intersects $\alpha_2$ in a single point we have a wedge of circles $\alpha= \alpha_1\vee \alpha_2$. 
Pushing $\alpha$ off of $F$ we get $\alpha^+$ and $\alpha^-$.  Using the Seifert framed longitudes, we get framings of $\alpha^+$ and $\alpha^-$ satisfying the assumptions of Proposition~\ref{string link mod}.  Let $J$ be a $2$-component string link with pairwise linking number $0$ for which $\overline{\mu}_{1122}(\widehat J) = \overline{\mu}_{1122}(L)$.  Then $K' := K_{\alpha^+, \alpha^-}(J,-{J})$ is homology concordant to $K$ with derivative $L':=L_{\alpha^+, \alpha^-}(J,-{J})$.  

Let $A = \begin{bmatrix}\lnk(L_1,\delta_1^+) &\lnk(L_1,\delta_2^+)\\
\lnk(L_2,\delta_1^+) &\lnk(L_2,\delta_2^+)\end{bmatrix}$ and $B=\begin{bmatrix}\lnk(L_1,\delta_1^-) &\lnk(L_1,\delta_2^-)\\\lnk(L_2,\delta_1^-) &\lnk(L_2,\delta_2^-)\end{bmatrix}$.
By Corollary~\ref{cor: infection and SL},  
$$
\begin{array}{rcll}
\overline{\mu}_{1122}(L') &\equiv& \overline{\mu}_{1122}(L) + \det(A)\overline{\mu}_{1122}(\widehat J)+\det(B) \overline{\mu}_{1122}( - \widehat J) & \mod2

\\&\equiv&\overline{\mu}_{1122}(L) + \det(A)\overline{\mu}_{1122}(\widehat J)-\det(B) \overline{\mu}_{1122}(\widehat J) & \mod2

\\&\equiv & \left(1+\det(A) - \det(B)\right)\overline{\mu}_{1122}(L) & \mod2.
\\
\end{array}
$$
 Since we assume $\det(A)-\det(B)$ is odd, $\overline{\mu}_{1122}(L')$ is even.  By Corollary~\ref{genus 2 step 1}, $K'$ is $1$-solvable.  Since $K$ is homology concordant to $K'$, $K$ is $1$-solvable. \end{proof}

Finally we prove Theorem \ref{main thm genus 2}.  

\begin{theorem}\label{main thm genus 2}
Let $K$ be an algebraically slice knot with genus $2$ Seifert surface $F$ and derivative $L=L_1,L_2$.   Extend $L_1,L_2$ to a basis for $H_1(F)$.  Let $M=\begin{bmatrix}0&A\\B&C\end{bmatrix}$ be the resulting Seifert matrix.

If either $\overline{\mu}_{1122}(L)$ is even or $\det(A)-\det(B)$ is odd, then $K$ is 1-solvable.  
\end{theorem}

\begin{proof}[Proof of Theorem~\ref{main thm genus 2}]

Let $K$ be a genus 2 algebraically slice knot with Seifert surface $F$ and derivative $L=L_1,L_2$.  If $\overline{\mu}_{1122}(L)$ is even then Corollary~\ref{genus 2 step 1} applies and $K$ is $1$-solvable, completing the proof in this case.  

Otherwise, extend $L$ to a full basis $\{L_1, L_2, \alpha_1,\alpha_2\}$ for $H_1(F)$ to get Seifert matrix $\begin{bmatrix}0&A\\B&C\end{bmatrix}$ where $0$, $A$, $B$, and $C$ are all $2\times 2$ matrices. According to Remark~\ref{basis} we may choose $\alpha_1$ and $\alpha_2$ to be disjoint simple closed curves. By assumption $\det(A)-\det(B)=\det(A)-\det(B^T)$ is odd.  Since $A = \begin{bmatrix}\lnk(L_1,\alpha_1^+) &\lnk(L_1,\alpha_2^+)\\
\lnk(L_2,\alpha_1^+) &\lnk(L_2,\alpha_2^+)\end{bmatrix}$ 
and 
$B^T 
= 
\begin{bmatrix}\lnk(L_1,\alpha_1^-) &\lnk(L_1,\alpha_2^-)\\
\lnk(L_2,\alpha_1^-) & \lnk(L_2,\alpha_2^-)\end{bmatrix}$, Proposition~\ref{Change iijj} concludes that $K$ is $1$-solvable. \end{proof}
\section{Higher genus results.}\label{sect: triple linking}

In the case of links of $3$ or more components, one additional invariant is needed to understand $0$-solvability, \textbf{Milnor's triple linking number}, $\overline\mu_{123}$.  

\begin{definition} Let $L=L_1,L_2,L_3$ be a $3$-component link with vanishing pairwise linking number. There exists Seifert surfaces $F_1,F_2,F_3$ for $L_1,L_2,L_3$ with transverse intersections such that for all $i\neq j$, $F_i\cap L_j=\emptyset$. Then $\overline{\mu}_{123}(L)$, is the number of points in $F_1\cap F_2\cap F_3$, known as triple points, counted with sign.
\end{definition}

There is a complete obstruction theory due to the second author to a link being $0$-solvable.

\begin{proposition*}[Theorem 1 of \cite{Martin15}]
Let $L=L_1,\dots, L_n$ be an $n$-component link.  Then $L$ is $0$-solvable if and only if 
\begin{enumerate}
\item $\lnk(L_i,L_j) = 0$ for $i\neq j$,
\item $\Arf(L_i)= 0 \mod 2$ for all $i$
\item For every $2$-component sublink $\overline\mu_{1122}(L_i,L_j)\equiv 0 \mod 2$, and 
\item For every $3$-component sublink $\overline{\mu}_{123}(L_i, L_j, L_k)=0$.
\end{enumerate}
\end{proposition*}

We recall the effect of a string link infection in terms of the Milnor's triple linking number due to the fourth author.  It is stated and proved in \cite{Park16} for classical string link operators and not generalized string link operators, but the proof makes no reference to this assumption.  A similar result appears also in \cite{JKP}.

\begin{proposition*}[Lemma 3.1 of \cite{Park16}] Let $L=L_1,L_2,L_3$ be an oriented $3$-component link with pairwise linking numbers zero.  Let $J$ be a $3$-component string link such that $\widehat{J}$ has pairwise linking numbers zero. Let $\alpha = \alpha_1\vee \alpha_2 \vee \alpha_3$ be a wedge of circles in the complement of $L$ and  $C = (c_{i,j})$ be the $3\times 3$ matrix given by $c_{i,j} = \lnk(L_i, \alpha_j)$.  Then 
$$\overline{\mu}_{123}(L_\alpha(J)) = \overline{\mu}_{123}(L) + \det(C)\cdot\overline{\mu}_{123}(\widehat{J}).$$
\end{proposition*}

\begin{remark}\label{rem: preserve triple} \normalfont Notice that if $L$ is a $3$-component link,  $\alpha = \alpha_1\vee \alpha_2$ is a wedge of two circles in the complement of $L$ and $J$ is a $2$-component string link, $J$ with zero pairwise linking, then 
$\overline{\mu}_{123}(L) = \overline{\mu}_{123}(L_{\alpha}(J))$.
This can be easily justified since the procedure of modifying the Seifert surfaces seen in the proof of Proposition~\ref{prop:infection and SL} does not introduce any new triple points in the intersection of the Seifert surfaces.
\end{remark}

Notice that infection by a $3$-component string link (even one with vanishing Sato-Levine invariant) might alter $\overline{\mu}_{1122}$.   For example, as seen in Figure~\ref{fig: whiteheadlink}, the Whitehead link can be obtained by infecting an unlink by the Borromean rings.  While it is possible to write down a formula revealing the precise effect of string link infection on the Sato-Levine invariant, the formula and proof are both sufficiently complicated to be prohibitive in their use.

\begin{figure}[h]
\setlength{\unitlength}{1pt}
\begin{picture}(0,290)
\put(-170,20){\includegraphics[height=.40\textheight]{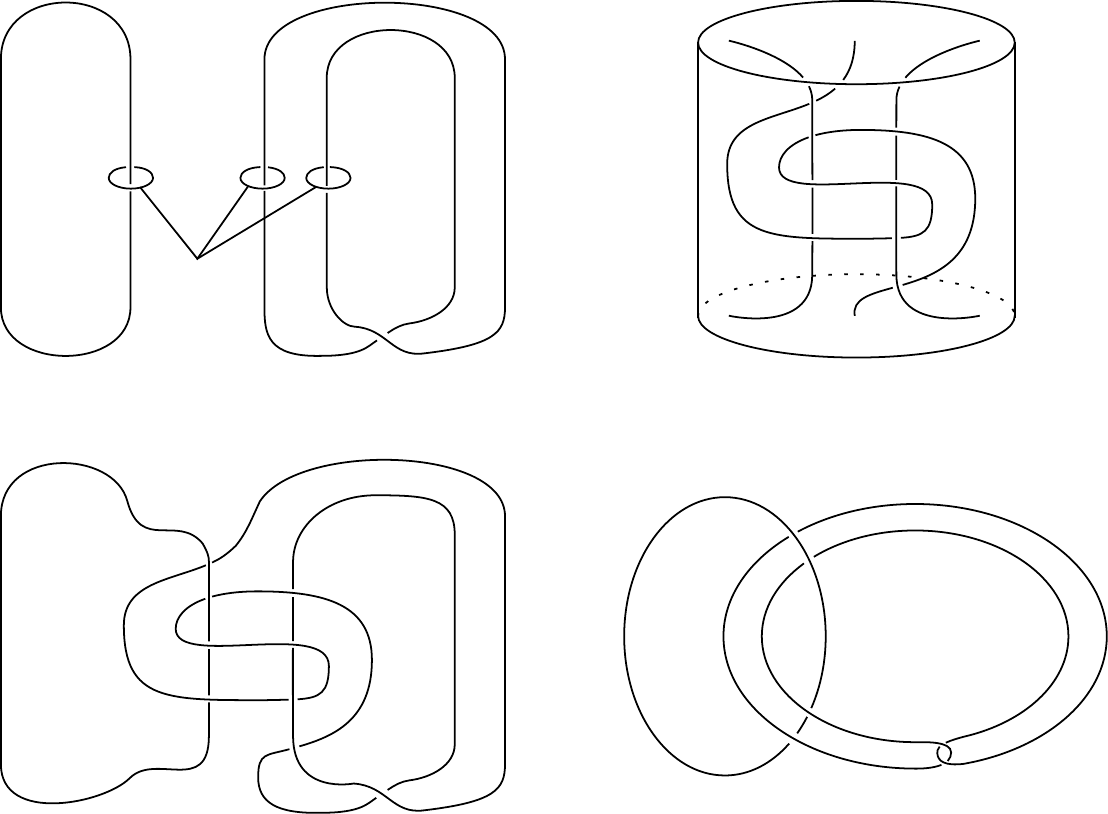}}
\put(-90,150){$L$}
\put(-100,190){$\alpha$}
\put(100,150){$J$}
\put(-100,-5){$L_\alpha(J)$}
\put(65,-5){Whitehead link}
\put(-210,70){$\longrightarrow$}
\put(7,70){$\cong$}
\end{picture}
\caption{The Whitehead link $L_\alpha(J)$ obtained from infecting the unlink $L$ by the Borromean rings $\widehat{J}$.}\label{fig: whiteheadlink}
\end{figure}
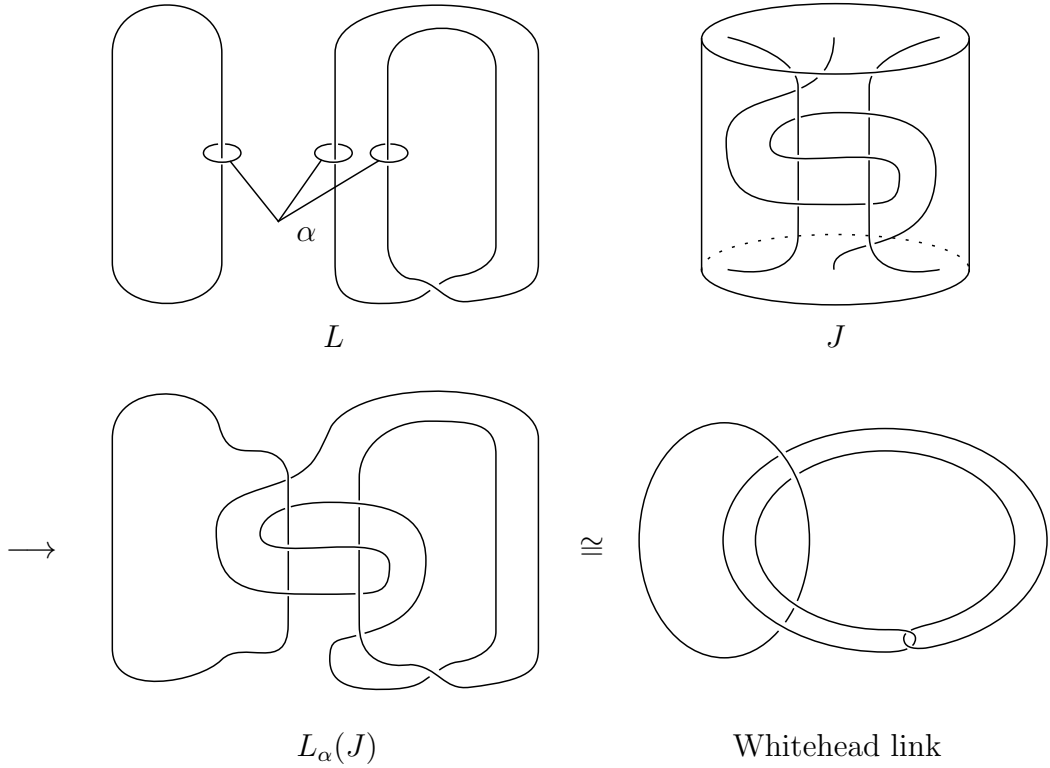

Before we provide the proof of Theorem~\ref{main thm genus 3}, we recall some notation that will be useful.  For a $k$-component link $L=L_1,\dots, L_k$, we will sum together the Sato-Levine invariants of all of the $2$-component sublinks of $L$.  Let 
$$
\SL(L) = \Sum_{1\le i<j\le k}\overline{\mu}_{1122}(L_i,L_j)e_i\wedge e_j\in (\Z/2)^k\wedge (\Z/2)^k.
$$ 

The reader will notice that $\SL$ only considers the Sato-Levine invariants modulo 2.  A little bit of algebra reveals how infection by a $2$-component string link changes $\SL(L)$.  


\begin{proposition}\label{infection and SL 2}
Let $L = L_1, \dots, L_k$ be a $k$-component link and $\alpha = \alpha_1\vee \alpha_2$ be a wedge of circles in the complement of $L$.  Let $J$ be a $2$-component string link with zero linking numbers.  Let $C = (c_{i,j})$ be the $k\times 2$ matrix given by $c_{i,j} = \lnk(L_i, \alpha_j)$.  Then 
$$
\begin{array}{rcl}\SL(L_\alpha(J)) &=& \SL(L)+(Ce_1\wedge Ce_2) \overline\mu_{1122}(\widehat{J}), \end{array}$$
where $e_1$ and $e_2$ are the standard basis vectors in $(\Z/2)^2$.  
\end{proposition}
\begin{proof}
Let $L'=L_\alpha(J)$.  Using Corollary~\ref{cor: infection and SL}, we see that 
$$
\begin{array}{rcl}\SL(L') &=& \Sum_{1\le i<j\le k} \overline{\mu}_{1122}(L_i', L_j') e_i \wedge e_j 
\\&=&\Sum_{1\le i<j\le k} \left(\overline{\mu}_{1122}(L_i, L_j)+\det \begin{bmatrix}c_{i,1}&c_{i,2}\\c_{j,1}&c_{j,2}\end{bmatrix}\overline\mu_{1122}(\widehat J)\right) e_i \wedge e_j 
\\&=& \SL(L)+\Sum_{1\le i<j\le k} \left(\det \begin{bmatrix}c_{i,1}&c_{i,2}\\c_{j,1}&c_{j,2}\end{bmatrix}\overline\mu_{1122}(\widehat J)\right) e_i \wedge e_j.
\end{array}$$
To complete the proof it suffices to show that $\Sum_{1\le i<j\le k} \det \begin{bmatrix}c_{i,1}&c_{i,2}\\c_{j,1}&c_{j,2}\end{bmatrix} e_i \wedge e_j = Ce_1\wedge Ce_2$. Expanding the right quantity, we obtain 
$$
\begin{array}{rcl}
Ce_1\wedge Ce_2&=&
\left(\Sum_{i=1}^k c_{i,1}e_i\right)\wedge\left(\Sum_{j=1}^k c_{j,2}e_j\right) = \Sum_{i=1}^k\Sum_{j=1}^k c_{i,1}c_{j,2}e_i\wedge e_j.
\end{array}
$$
Using the facts that $e_i\wedge e_i=0$ and $e_j\wedge e_i=-e_i\wedge e_j$, we can reorder this sum.
$$
\begin{array}{rcl}
Ce_1\wedge Ce_2&=& \Sum_{1\le i<j\le k} \left(c_{i,1}c_{j,2}-c_{i,2}c_{j,1}\right)e_i\wedge e_j = \Sum_{1\le i<j\le k} \det \begin{bmatrix}c_{i,1}&c_{i,2}\\c_{j,1}&c_{j,2}\end{bmatrix} e_i \wedge e_j.
\end{array}
$$
The second equality is simply the definition of the determinant.  This completes the proof. \end{proof}

Even though Theorem~\ref{main thm genus 3} is a consequence of Theorem~\ref{main thm genus g}, a proof for Theorem~\ref{main thm genus 3} is provided now.

\begin{theorem}\label{main thm genus 3}
Let $K$ be an algebraically slice knot with genus $3$ Seifert surface $F$ and derivative $L=L_1,L_2, L_3$.   Extend $L_1,L_2, L_3$ to a basis for $H_1(F)$.  Let $M=\begin{bmatrix}0&A\\B&C\end{bmatrix}$ be the resulting Seifert matrix.

If $\overline\mu_{123}(L)\in \Z$ is a multiple of $(\det(A)-\det(B))$ and $(A\wedge A-B^T\wedge B^T):(\Z/2)^3\wedge (\Z/2)^3\to(\Z/2)^3\wedge (\Z/2)^3$ is onto then $K$ is $1$-solvable.  
\end{theorem}

\begin{proof}[Proof of Theorem~\ref{main thm genus 3}]

Let $K$ be an algebraically slice knot with genus 3 Seifert surface $F$ and derivative $L=L_1,L_2,L_3$.  Extend $L_1,L_2,L_3$ to a basis $\{L_1,L_2,L_3, \alpha_1,\alpha_2,\alpha_3\}$ for $H_1(F)$ to obtain the Seifert matrix $M=\begin{bmatrix}0&A\\B&C\end{bmatrix}$.  According to Remark~\ref{basis} we may choose $\alpha_1, \alpha_2, \alpha_3$ to be disjoint simple closed curves. By assumption, $\overline\mu_{123}(L) =  (\det(A)-\det(B))k$ for some $k\in \Z$.  Let $J$ be any $3$-component string link with $\overline\mu_{123}(\widehat{J})=-k$.

Using arcs to connect $\alpha_1$, $\alpha_2$, and $\alpha_3$ results in a wedge of circles $\alpha=\alpha_1\vee\alpha_2\vee\alpha_3$.  By Proposition~\ref{string link mod} 
we see that $K':=K_{\alpha^+,\alpha^-}(J,-J)$ is homology concordant to $K$.

On the genus 3 Seifert surface $F'$ resulting from this infection,  $K'$ has as a derivative $L':=L_{\alpha^+,\alpha^-}(J,-J)$.  Notice that the $(i,j)$-entry of $A$ is $A_{i,j} = \lnk(L_i,\alpha_j^+)$ and the $(j,i)$-entry of $B$ is $B_{j,i} = \lnk(\alpha_j,L_i^+) = \lnk(L_i, \alpha_j^-)$.   According to \cite[Lemma 3.1]{Park16}, and our assumptions regarding the triple linking numbers of $L$ and $\widehat{J}$, we have
$$
\begin{array}{rcl}
\overline\mu_{123}(L') &=& \overline\mu_{123}(L)+\det(A)\overline\mu_{123}(\widehat J) - \det(B^T)\overline\mu_{123}(\widehat J) 
\\&=& (\det(A)-\det(B))k+(\det(A)-\det(B))(-k)
\\&=&0.
\end{array}
$$
We no longer have any control over $\SL(L')$.  Instead, since we assume $A\wedge A-B^T\wedge B^T$ is onto on $(\Z/2)^3\wedge (\Z/2)^3$, there exist $x,y,z\in \Z/2$ such that,   
$$
\SL(L') = \left(A\wedge A-B^T\wedge B^T\right)(x\cdot e_1\wedge e_2+y\cdot e_1\wedge e_3+z\cdot e_2\wedge e_3).
$$
  Let $\beta_1,\beta_2,\beta_3$ be the image of $\alpha_1,\alpha_2,\alpha_3$ on $F'$.  Let $\beta = \beta_1\vee \beta_2$ and $X$ be some $2$-component string link with $\overline\mu_{1122}(\widehat{X})= x$.  Let $K'' = K'_{\beta^+,\beta^-}(X,-X)$.  Then $K''$ has derivative $L''=L'_{\beta^+,\beta^-}(X,-X)$ such that $\SL(L'')$ is given by
$$ \left(A\wedge A-B^T\wedge B^T\right)(x\cdot e_1\wedge e_2+y\cdot e_1\wedge e_3+z\cdot e_2\wedge e_3) -  x\cdot \left(A\wedge A-B^T\wedge B^T\right)(e_1\wedge e_2)$$
$$=\left(A\wedge A-B^T\wedge B^T\right)(y\cdot e_1\wedge e_3+z\cdot e_2\wedge e_3).
$$

According to Remark~\ref{rem: preserve triple},  $\overline\mu_{123}(L'') = \overline\mu_{123}(L')=0$. Similarly, we infect along the image of $\alpha_1\vee \alpha_3$ and $\alpha_2\vee \alpha_3$ to eliminate $y$ and $z$.  Thus, $K$ is homology concordant to a genus 3 knot $K'''$ admitting derivative $L'''$ with $\overline\mu_{123}(L''')=0$ and $\SL(L''')=0$ over $\Z/2$.

Utilizing Proposition~\ref{main thm genus 1'} we further modify $K$ until the $\Arf$-invariants of the components of the derivative vanish.  Thus, $K$ is homology concordant to a knot which admits a $0$-solvable derivative, and so is $1$-solvable. \end{proof}

Recall the definition from the introduction, for an $n$ component link  $L$
$$\TL(L) = \Sum_{1\le i<j<k\le n}\overline{\mu}_{123}(L_i,L_j,L_k)e_i\wedge e_j\wedge e_k\in \Z^n\wedge \Z^n\wedge \Z^n.
$$

The exact same proof used to prove Proposition~\ref{infection and SL 2} using \cite[Lemma 3.1]{Park16} in place of Corollary~\ref{cor: infection and SL} shows that 
\begin{proposition}\label{infection and TL}
Let $L = L_1, \dots, L_n$ be an $n$-component link.  Let $\alpha = \alpha_1\vee \alpha_2\vee \alpha_3$ be a wedge of three circles in the complement of $L$.  Let $J$ be a $3$-component string link with zero pairwise linking numbers.  Let $C = (c_{i,j})$ be the $n\times 3$ matrix given by $c_{i,j} = \lnk(L_i, \alpha_j)$.  Then 
$$
\begin{array}{rcl}\TL(L_\alpha(J)) &=& \TL(L)+(Ce_1\wedge Ce_2\wedge Ce_3) \overline\mu_{123}(\widehat{J}) \end{array}$$
where $e_1$, $e_2$, and $e_3$ are the standard basis vectors in $\Z^3$.  
\end{proposition}

The proof of Theorem~\ref{main thm genus g} below uses the same ideas as that of Theorem~\ref{main thm genus 3}.

\begin{theorem}\label{main thm genus g}
Let $K$ be an algebraically slice knot with genus $g$ Seifert surface $F$ and derivative $L=L_1, \dots, L_g$.   Extend $L_1,\dots, L_g$ to a basis for $H_1(F)$.  Let $M=\begin{bmatrix}0&A\\B&C\end{bmatrix}$ be the resulting Seifert matrix.  If $\TL(L)\in \Z^g\wedge\Z^g\wedge\Z^g$ is in the image of $A\wedge A \wedge A-B^T \wedge B^T \wedge B^T$ and $A\wedge A-B^T\wedge B^T:(\Z/2)^g\wedge (\Z/2)^g\to(\Z/2)^g\wedge (\Z/2)^g$ is onto, then $K$ is $1$-solvable.  
\end{theorem}

\begin{proof}

Let $K$ be an algebraically slice knot with genus $g$ Seifert surface $F$ and derivative $L=L_1,\dots,L_g$.  Extend $L_1,\dots,L_g$ to a basis $\{L_1,\dots,L_g, \alpha_1,\dots,\alpha_g\}$ for $H_1(F)$ to obtain the $2g\times 2g$ Seifert matrix $M=\begin{bmatrix}0&A\\B&C\end{bmatrix}$. According to Remark~\ref{basis} we may choose $\alpha_1 \dots \alpha_g$ to be disjoint simple closed curves. By assumption, $\TL(L)$ is in the image of $A\wedge A \wedge A-B^T \wedge B^T \wedge B^T$.  For some choice of $x_{ijk}\in \Z$
$$
\TL(L) = \left(A\wedge A \wedge A-B^T \wedge B^T \wedge B^T\right)\left(\Sum_{1\le i<j<k\le g} x_{ijk}e_i\wedge e_j\wedge e_k\right).
$$

Let $\alpha$ be the wedge of circles $\alpha_1\vee \alpha_2\vee \alpha_3$ and $X$ be a $3$-component string link with $\overline\mu_{123}(\widehat{X}) = -x_{123}$.  According to Proposition~\ref{string link mod}, $K$ is homology concordant to $K':=K_{\alpha^+,\alpha^-}(X,-X)$.  On $F'$, the Seifert surface resulting from this infection, $K'$ admits a derivative $L' = L_{\alpha^+,\alpha^-}(X,-X)$.   

According to Proposition~\ref{infection and TL} and the assumptions about $\TL(L)$ and $\overline\mu_{123}(\widehat X)$,
$$
\begin{array}{rcl}
\TL(L') &=&\TL(L) +  (Ae_1\wedge Ae_2\wedge Ae_3)\overline\mu_{123}(\widehat X)- (B^Te_1\wedge B^Te_2\wedge B^Te_3)\overline\mu_{123}(\widehat X)
\\ &=& \left(A\wedge A \wedge A-B^T \wedge B^T \wedge B^T\right)\left(\Sum_{1\le i<j<k\le g} x_{ijk}e_i\wedge e_j\wedge e_k\right)
 \\ & & - x_{123}\left(A\wedge A \wedge A-B^T \wedge B^T \wedge B^T\right) e_1\wedge e_2\wedge e_3
\end{array}
$$

Thus, $$\TL(L') = \left(A\wedge A \wedge A-B^T \wedge B^T \wedge B^T\right)\left(\Sum_{1\le i<j<k\le n} x'_{ijk}e_i\wedge e_j\wedge e_k\right)$$ where $x'_{123}=0$ and otherwise $x'_{ijk}=x_{ijk}$.  Repeating this process using all remaining choices of $i,j,k$ results in a genus $g$ knot $K''$ homology concordant to $K$ admitting a derivative $L''$ with $\TL(L'')=0$. 

We have no control over $\SL(L'')$.  Instead we use the assumption that $A\wedge A-B^T\wedge B^T$ is onto on $(\Z/2)^g \wedge (\Z/2)^g$ to conclude that for some $y_{ij}$
$$
\SL(L'') = (A\wedge A-B^T\wedge B^T)\Sum_{1\le i<j\le g} y_{ij} e_i\wedge e_j.
$$

By infecting along $\alpha_i\vee \alpha_j$ we can replace $L''$ with a new derivative $L'''$ for which $y_{ij} = 0$.  Thus, $K$ is homology concordant to a knot $K'''$ with derivative $L'''$ for which $\TL(L''')$ and $\SL(L''')$ both vanish.  Finally we use Proposition~\ref{main thm genus 1'} to get a fourth knot $K^{(4)}$ homology concordant to $K$ which admits a derivative $L^{(4)}$ for which $\TL(L^{(4)})$, $\SL(L^{(4)})$, and the Arf-invariant of every component vanish.  According to \cite[Theorem 1]{Martin15} $L^{(4)}$ is $0$-solvable so that $K^{(4)}$ is 1-solvable. Since $K$ is homology concordant to a $1$-solvable knot, $K$ is $1$-solvable. This completes the proof. \end{proof}

\end{document}